\documentclass[12pt]{amsart}
\usepackage{amssymb, amscd, amsmath, amsthm, amscd, 
epsfig, latexsym, enumerate}

\renewcommand{\geq}{\geqslant}
\renewcommand{\leq}{\leqslant}

\newtheorem{theorem}{Theorem}
\newtheorem{lemma}[theorem]{Lemma}
\newtheorem{corollary}[theorem]{Corollary}

\theoremstyle{definition}

\newtheorem{example}[theorem]{Example}

\newtheorem{maintheorem}{Theorem}

\newcommand{\F}{\mathbb F}

\newcommand{\Z}{\mathbb Z}

\DeclareMathOperator{\Aut}{Aut}
\DeclareMathOperator{\cd}{cd}
\DeclareMathOperator{\hd}{hd}
\DeclareMathOperator{\Epi}{Epi}
\DeclareMathOperator{\Ext}{Ext}
\DeclareMathOperator{\Hom}{Hom}
\DeclareMathOperator{\Homeo}{Homeo}

\DeclareMathOperator{\inc}{inc}
\DeclareMathOperator{\rank}{rank}

\newcommand{\ol}[1]{{\overline{#1}}}

\newcommand{\wh}[1]{{\widehat{#1}}}
\newcommand{\wt}[1]{{\widetilde{#1}}}

\begin{document}

\title[Aspherical 4-manifolds]{Aspherical 4-manifolds with elementary
amenable fundamental group}

\author{James F.  Davis \and J. A. Hillman }
\address{{Department of Mathematics, Indiana University,}
\newline
{Bloomington, IN 47405 USA} 
\newline
{School of Mathematics and Statistics, University of Sydney,}
\newline
Sydney,  NSW 2006, Australia }

\email{jfdavis@iu.edu, jonathan.hillman@sydney.edu.au}

\begin{abstract}
We classify the possible elementary amenable fundamental groups of compact aspherical  4-manifolds 
with boundary   and conclude that they are either polycyclic 
or solvable Baumslag-Solitar.    Since these groups are good and satisfy the Farrell-Jones Conjecture, one concludes that such manifolds satisfy topological rigidity: a homotopy equivalence which is a homeomorphism on the boundary is homotopic, relative to the boundary, to a homeomorphism.  We classify the closed 3-manifolds which arise as the boundary of a compact aspherical 4-manifold with elementary amenable fundamental group, generalizing results of Freedman and Quinn in the cases of trivial and infinite cyclic fundamental groups.
Moreover, two such 4-manifolds are homeomorphic if and only if 
their ``enhanced" peripheral group systems are equivalent,
and each such manifold is the boundary connected sum of a compact
aspherical 4-manifold with prime boundary and a
contractible 4-manifold. 

\end{abstract}



\maketitle

The goal of this paper is the classification of compact aspherical 4-manifolds 
with elementary amenable fundamental group and nonempty boundary.
Our focus on such manifolds reflects the following considerations.
Firstly,  aspherical manifolds are determined up to homotopy equivalence 
by their fundamental groups, and these groups have finite cohomological dimension.
Secondly, the classification of topological 4-manifolds with given homotopy type 
involves applying high-dimensional techniques such as surgery and the s-cobordism theorem, and elementary amenable groups form the largest well-understood class of 
groups over which these techniques are known to work in dimension 4
and which have finite cohomological dimension.

More precisely,  if a group $\pi$ is good, in the sense of \cite[Definition 12.12]{BKKPR},
then the Disc Embedding Conjecture (DEC) holds for 4-manifolds 
with fundamental group $\pi$.   
Surgery and the s-cobordism theorem are consequences of the DEC, 
and hence are valid for manifolds with good fundamental group.  
If $\pi$ is the fundamental group of a compact aspherical 4-manifold, 
then $\pi$ is currently known to be good if and only if $\pi$ is elementary amenable.  This follows from the result of Freedman-Quinn (see \cite[Part III]{BKKPR}) that elementary amenable groups are good, the result of Freedman-Teichner \cite{FT95}  and Krushkal-Quinn \cite{KQ} that groups with subexponential growth are good, and the fact that  there is no finitely presented group of intermediate growth  known which has finite cohomological dimension (see acknowledgements).
  
The above discussion motivates our first theorem.
  
 \begin{maintheorem} \label{A}
Let $M$ be a compact aspherical $4$-manifold with boundary $N$,
and let $\inc:N\to{M}$ be the inclusion.
If $\pi=\pi_1M$ is elementary amenable, then one of the following conditions holds:
\begin{enumerate}
 \setcounter{enumi}{-1}
\item$\pi=1$ and $N$ is a homology $3$-sphere;
\item$\pi\cong\mathbb{Z}$, $N$ is connected, and $\pi_1\inc$ 
is an epimorphism with perfect kernel 
(i.e.,  $N$ is a $\Z[\Z]$-homology $S^2\times{S^1}$);
\item$\pi\cong{BS(1,m)}$ for some $m\not=0$, $N$ is connected, 
$\pi_1\inc$ is an epimorphism and $H^i(\inc;\Z\pi)$ is an isomorphism
for $i\leq2$; 
\item$\pi$ is a polycyclic $PD_3$-group and either 
\begin{enumerate}
\item$N$ has 
two components and $\pi_1\inc_i$ is an epimorphism for $i=1,2$;
or 
\item$N$ is connected and $[\pi:\mathrm{Im}(\pi_1\inc)]=2$;
\end{enumerate}
 in each case the kernels are perfect; 
\item$\pi$ is a polycyclic $PD_4$-group and $N$ is empty.
\end{enumerate}
\end{maintheorem}

Thus the fundamental group of a compact aspherical 4-manifold with elementary amenable fundamental group is either polycyclic or a Baumslag-Solitar group $BS(1,m)$.  
The Farrell-Jones Conjecture (FJC) in $K$- and $L$-theory holds for the groups
allowed by Theorem \ref{A}, so not only can we use topological surgery and the s-cobordism theorem, but we can also compute the obstruction groups.

The Farrell-Jones Conjecture and topological surgery implies the Borel Uniqueness Conjecture for a compact aspherical 4-manifold $M$ with elementary amenable fundamental group:  any homotopy equivalence $h : M' \to M$ from a compact manifold $M'$ which restricts to a homeomorphism $h : \partial M' \xrightarrow{\cong} \partial M$ on the boundary is homotopic to a homeomorphism with the homotopy fixed on the boundary.  To obtain a full classification, we need an existence theorem: the converse to Theorem \ref{A} holds, and characterizes the possible boundaries of compact aspherical 4-manifolds with elementary amenable fundamental group:

\begin{maintheorem} \label{B}
Let $\pi$ be a finitely presented, elementary amenable group and $N$ be a closed $3$-manifold.
If $p:N\to{K(\pi,1)}$ is a map such that $\pi$, $N$ and $p$ satisfy
one of the conditions of Theorem \ref{A} (with $p$ for $\inc$) 
then there is a compact aspherical $4$-manifold $M$ with $\pi_1M\cong\pi$ and
boundary $N$, and a homotopy equivalence $h:M\to K(\pi,1)$ such that
$p=h\circ\inc$, where $\inc:N\to{M}$ is the inclusion.
\end{maintheorem}

The closed case of Theorem \ref{B}, corresponding to case (4) of Theorem \ref{A},
was settled in \cite{AJ}.
When $N$ is nonempty,  
Theorem \ref{B} is an immediate consequence of a result from \cite{dh1} on 
realizing such maps $p$ by the boundary inclusions of $PD_4$-pairs
(see Theorem \ref{PD4pair} below),
and on the Borel Existence Conjecture for pairs with $\pi$ elementary amenable \cite{dh1}. 

\begin{maintheorem} \label{C}
Let $M$ and $\widehat{M}$ be compact aspherical $4$-manifolds with elementary amenable fundamental groups
and orientable boundaries.
Then $M$ and $\widehat{M}$ are homeomorphic if and only if
their enhanced peripheral systems are equivalent.
\end{maintheorem}
 
A enhanced peripheral system consists of the fundamental group with orientation character,
the homomorphisms on $\pi_1$ induced by the inclusions 
of the boundary components, and data on the fundamental classes.
A full definition is given in \S\ref{section:uniqueness}.

A consequence of Theorem \ref{C} is that a compact aspherical 4-manifold with elementary amenable fundamental group 
and orientable boundary 
is the boundary connected sum of a compact aspherical 4-manifold 
with prime boundary and a contractible 4-manifold.
The orientability condition is need  because the Geometrization Theorem of 
Perelman and Thurston has not yet been proven for nonorientable 
3-manifolds.

In the closed case, $\pi$ is a polycyclic $PD_4$-group,
and the homeomorphism types of such 4-manifolds are classified  
by their fundamental groups.
Theorems B and C lead quickly to a similar classification for the bounded case,
which we shall state here in a simplified form, for the cases with connected boundary.

\begin{maintheorem} \label{D}
Compact aspherical 4-manifolds with elementary amenable fundamental group $\pi$ 
and connected nonempty boundary $N$ are classified up to homeomorphism by 
the isomorphism type of $\pi$, 
the homeomorphism type of $N$ and the equivalence class 
of the homomorphism $\pi_1\inc:\pi_1N\to\pi$ under the actions of $\Homeo(N,*)$ and $\Aut(\pi)$.  Every set of such invariants satisfying the conditions of Theorem \ref{A}
is realizable. 
\end{maintheorem}

What makes our results possible is the conjunction of 3-manifold theory 
(particularly the unique factorization theorem and the rigidity of closed orientable 
prime 3-manifolds), 
topological surgery (particularly in the Ranicki formulation)
and low-dimensional homological group theory.
The arguments extend without essential change to all groups $\pi$ with cohomological dimension
$\cd\pi\leq2$ and for which the Disc Embedding Conjecture (DEC) and 
Farrell-Jones Conjectures hold.
The DEC holds for all elementary amenable groups,
and the most optimistic current expectation is that it might hold 
for all groups with no noncyclic free subgroups.
In \S7 we give some evidence to suggest that groups with 
the latter property and $\cd\leq2$ may all be solvable.
Many of our arguments can easily by made in greater generality than elementary amenable,
and we shall do so where possible, without further comment.
Only in the case of groups $\pi$ with $\cd\pi=3$ does the hypothesis
``elementary amenable" seem to be a limitation which is essential for our arguments.

We feel that the study of compact aspherical manifolds is fundamental in four-dimensional topology in both the topological and smooth case.   The contractible case ($\pi_1 =1$) has been a persistent theme in the subject.   The questions of topological rigidity for a product of surfaces and the smooth rigidity for the 4-torus are two major open problems.   Our paper classifies the good fundamental groups of compact aspherical manifolds $M$ and all the possible 3-manifold boundaries.   The paper \cite{DKP} shows that topologically classifying all compact manifolds with fundamental group $\pi_1M$ and boundary $\partial M$ is a tractable problem, generalizing the famous results in the simply-connected case and infinite cyclic fundamental group case.

Our paper is organized as follows.
In \S1 we present the notation and terminology for the groups that arise here.
The next section summarizes the basic properties of the peripheral systems 
of compact aspherical manifolds.
These are standard consequences of Poincar\'e-Lefshetz duality.
(The details of the proofs are given in \cite{dh3}.)
We then focus on dimension 4 in \S3, where we prove Theorem \ref{A}.
We prove the main existence and uniqueness results, Theorems B and C, 
in \S4 and \S5 respectively.  
We also prove Lemma \ref{ZHS summand} which analyzes the 3-manifold boundaries.
In many cases (but not all) the compact aspherical 4-manifold $M$ is determined 
up to homeomorphism by $\pi_1M$ and $\partial{M}$. 
We give examples showing that further invariants are need in \S6,
and we prove Theorem \ref{D} (in a more precise formulation) there.
The  final section considers briefly possible extensions of this work.

\section{Group theoretic notation and terminology}

The class of {\em elementary amenable groups} is the smallest class of groups containing all finite groups and all abelian groups and which is closed under subgroups, quotients, extensions, and directed colimits.  
The elementary amenable groups that arise in Theorem \ref{A} 
are either polycyclic or are solvable Baumslag-Solitar groups. 
We will assume basic facts about polycyclic groups; 
a suitable reference is \cite[Chapter 5]{Ro}.  
The number of infinite cyclic quotients in a subnormal series for a virtually polycyclic group 
$G$ is called the {\em Hirsch length} $h(G)$ \cite[5.4.13]{Ro}.
The notion of Hirsch length extends readily to virtually solvable groups
\cite[page 407]{Ro}, 
and has been extended further to all elementary amenable groups
\cite[\S1.5]{HiF}.
The solvable Baumslag-Solitar groups are the groups $BS(1,m)$
with presentation $\langle{a,t}\mid{tat^{-1}=a^m}\rangle$,
for some $m\not=0$.
These groups are semidirect products $\Z[1/m] \rtimes \Z$ and include the fundamental group of the torus $\mathbb{Z}^2=BS(1,1)$ and of the Klein bottle $BS(1,-1)$. 

A group $G$ has one end if and only if $G$ is infinite and $H^1(G; \Z G) = 0$ (see \cite{DD89}).   
Nontrivial torsion-free elementary amenable groups other than $\Z$ 
have one end.
The {\em cohomological dimension} of $G$ is 
$$
\cd G = \sup~\{n \mid H^n(G; A) \not = 0 \text{ for some $\Z{G}$-module $A$}. \}
$$
There is a similar notion of homological dimension $\hd G$,
and $\hd G\leq\cd G\leq\hd G+1$ if $G$ is countable \cite[4.6]{Bieri}.
If $G$ is elementary amenable and $\cd G<\infty$ then $G$ is virtually solvable \cite{HL92}.
If $G$ is torsion-free and solvable, 
with Hirsch length $h(G)$,
then $\hd G=h(G)$  \cite[7.1]{Bieri}.

A group $G$ has type $FP$ if the augmentation left $\Z G$-module $\Z$ has 
a finite length resolution by finitely generated projective $\Z G$-modules.
The Euler characteristic  $\chi(G)=\sum_{i\geq0}(-1)^i\rank H_iG$
is then well-defined.
The group $G$ is a {\em duality group} of dimension $n$
if it has type $FP$ and $H^k(G; \Z G)=0$  for $k \not = n$.
A duality group is torsion-free and has cohomological dimension $n$.  

A solvable group is {\em constructable} if it can be built up 
from the trivial group by a finite sequence of finite extensions 
and ascending HNN extensions \cite{BB76}.
If $S$ is solvable and $\cd S<\infty$ then $\cd S=\hd S\Leftrightarrow{S}$ has type 
$FP\Leftrightarrow{S}$ is a duality group $\Leftrightarrow{S}$ is constructable \cite{Kr86}.
In particular,  if $S$ is constructable then $\cd S=h(S)$ (by  \cite[7.1]{Bieri} again).

A group $\pi$ is a $PD_n$-group (a {\em Poincar\'e duality group} 
of dimension $n$) if 
$$
H^k(\pi; \Z\pi) =
\begin{cases}
 \Z & k = n \\
 0 & k \not  = n
\end{cases}
$$
and $\pi$ has type $FP$.
Since $\Z\pi$ is a $(\Z\pi,\Z\pi)$-bimodule, 
$H^n(\pi; \Z\pi)$ is a right $\Z\pi$-module.   
The corresponding homomorphism $w: \pi \to \Z^\times$ is called
the {\it orientation character\/} of the Poincar\'e duality group $\pi$.
Let $\Z^w$ be the $(\Z\pi,\Z\pi)$-bimodule which is additively the infinite cyclic group 
but with  $g . n = n.g =  w(g) n$,
and let $\varepsilon_w:\Z\pi\to\Z^w$ be the additive extension of $w$ 
(the {\it $w$-twisted augmentation homomorphism}).
Then $H^n(\pi; \Z\pi) =\Z^w$ as right $\Z\pi$-modules.

A $PD_n$-group $\pi$ satisfies Poincar\'e duality,
in that $H^k(\pi;\Z\pi)$ is isomorphic to $H_{n-k}(\pi;\Z\pi)$ as an abelian group,
for all $k\geq0$.
However, if we wish an isomorphism of (left) $\Z\pi$-modules, 
we need to replace $H^k(\pi;\Z\pi)$ by the conjugate module,
defined in terms of $w$.
If $B$ is a right $\Z\pi$-module, let $\ol{B}$ be the left module with the same underlying group and $\Z\pi$-module structure determined by $g.b=w(g)rg^{-1}$, for all $b\in{B}$ and $g\in\pi$. 
If $L$ is a left $\Z\pi$-module then $\Hom_{\Z\pi}(L,\Z\pi)$ is a right module, and so the {\it conjugate dual}
$L^\dagger=\ol{\Hom_{\Z\pi}(L,\Z\pi)}$ is again a left module.
If $M$ is a finitely generated free left $\Z\pi$-module then $M^\dagger$ is free, and is {\it non}-canonically isomorphic to $M$.
There are Poincar\'e duality isomorphisms
$\ol{H^k(\pi;\Z\pi)}\cong H_{n-k}(\pi;\Z\pi)$ (see \cite{dh1}).
Poincar\'e duality groups were first studied by Johnson and Wall, 
and duality groups are due to Bieri and Eckmann; 
a basic reference is the book of Brown \cite{Brown}.

If $\rho$ is a subgroup of finite index in a torsion-free group $\pi$
then $\rho$ is a $PD_n$-group if and only if $\pi$ is.  
Solvable Poincar\'e duality groups are polycyclic \cite{Bieri}, a torsion-free polycyclic group of Hirsch length $n$ is a $PD_n$-group, and, in fact, is the fundamental group of a closed aspherical smooth $n$-manifold \cite{AJ}.
The solvable Baumslag-Solitar groups $BS(1,m), m \not = 0$ have cohomological dimension two and,  
for $m \not = \pm 1$,
they are neither polycyclic nor Poincar\'e duality groups.  

\section{General remarks on the peripheral systems of aspherical manifolds}

Basic invariants of a manifold with boundary include its fundamental group, 
the number of boundary components, 
and the fundamental groups of the boundary components.     
We formalize this by introducing the notion of a peripheral system and 
then study this applied to aspherical manifolds.
Note that we will always assume our manifolds with boundary are connected, 
although, of course, we will not assume the boundary is connected.

Let $M$ be a compact $n$-manifold with boundary
$\partial{M}=\coprod_{i=1}^k\partial_i{M}$.  
Let $\inc_i : \partial_i{M} \to M$ be the inclusion of the $i$-th boundary component.   
(If $k=1$ we shall drop the label $i=1$.) 
Choosing base points $x_0 \in M$ and $x_i \in \partial_iM$ and paths $\gamma_i$ from $x_i$ to $x_0$, define the {\em fundamental group system of $M$} to be the tuple 
$(\pi_1M , \pi_1 \partial_1M, \dots, \pi_1 \partial_kM, \pi_1\inc_1, \dots , \pi_1\inc_k)$.   
The {\em orientation character} of $M$ is the homomorphism 
$w =w_1M  : \pi_1M \to \Z^{\times}= \{\pm 1\}$.  
The {\em peripheral system of $M$} is the fundamental group system together with the orientation character.

An isomorphism $(G, G_1, \dots, G_k, j_1: G_1 \to G, \dots, j_k : G_k \to G, w) \to (G', G'_1, \dots, G'_k, j'_1: G'_1 \to G', \dots, j'_k : G'_k \to G',w') $ of peripheral systems consists of a permutation $\sigma \in S_k$, group isomorphisms $\theta : G \to G'$, $\theta_i : G_i \to G'_{\sigma(i)}$, and elements $g_i \in G$ so that $w= w' \circ \theta$ and $\theta_i \circ j_i \circ c_{g_i} = j'_i \circ \theta_{\sigma(i)}$ where $c_{g_i}$ is conjugation by $g_i$.   The isomorphism class of the peripheral system is an invariant of $M$, independent of the choice of base points and paths.

A  group is {\em of type $F$} it is the fundamental group of a finite aspherical complex.
If $M$ is a compact aspherical $n$-manifold then $\pi=\pi_1M$ is of type $F$, 
and if $M$ is closed then $\pi$  is a $PD_n$-group with orientation character $w=w_1M$.
If $\partial{M}$ is nonempty then $\pi$ is no longer a $PD_n$-group, 
and $w_1M$ may not be determined by $\pi$.
In this case there are Poincar\'e-Lefshetz duality isomorphisms
$\overline{H^i(M;\Z\pi)}\cong{H_{n-i}(M,\partial{M};\Z\pi)}$,
where the conjugation uses the anti-involution induced by $w=w_1M$.

The first two results are straightforward consequences of Poincar\'e-Lefshetz duality 
with coefficients $\mathbb{F}_2$,  $\Z$ or $\Z\pi$,
for aspherical manifolds with boundary.
See \cite{dh3} for proofs of Lemma \ref{components of boundary}
and Theorem \ref{cd<n}.

\begin{lemma} \label{components of boundary}
Let $M$ be a compact aspherical $n$-manifold with fundamental group $\pi$.
\label{boundary components}
\begin{enumerate}
\item \label{cdpin} $\cd \pi =n \Longleftrightarrow \partial M$ is empty.
\item \label{cdpin-2}$\cd \pi \leq{n-2} \Longrightarrow \partial M$ has one component.
\item \label{pdn-1} $ \pi$ is $PD_{n-1} \Longrightarrow \partial M$ has one or two components.
\qed
\end{enumerate}
\end{lemma}

An example to keep in mind is $X^k\times D^{n-k}$, where $X$ is a closed aspherical manifold.      
Another example is a compact contractible manifold with boundary 
a homology sphere.   
Note that the boundary components of an aspherical manifold need not be aspherical.   
A compact aspherical manifold can have many boundary components, 
for example the cartesian product of 
a torus and a 2-sphere with many open disks removed.

Now we move from the discussion of the number of boundary components to 
their fundamental groups. 
We shall write $H_i(X)$ and $H_i(G)$ for the integral homology of a space $X$ or group $G$, for simplicity.

\begin{theorem}
\label{cd<n}
Let $M$ be a compact aspherical $n$-manifold.
Then $\pi=\pi_1M$ is of type $F$ and $\cd\pi\leq4$.
\begin{enumerate}
\item If $\cd \pi < n$ and $\pi \not = 1$, then the image of $\pi_1\inc_i$ is infinite for all $i$.
\item If $\cd \pi \leq{n-2}$, then $\pi_1\inc$ is an epimorphism.
\item If $\cd \pi =n- 2$, then $\pi_1\inc$ has nontrivial kernel.
\item If $\cd \pi \leq1$, then  $\ker(\pi_1\inc)$ is a perfect group. 
\item If $\pi$ is a $PD_2$-group, then $\ker(\pi_1\inc)$ 
has abelianization  $\Z$.
\item If $\pi$ is a $PD_{n-1}$-group with orientation character $w_{\pi}$,
then each $\ker(\pi_1\inc_i)$ is a perfect group and either
\begin{enumerate}
\item $\partial M$ has two components and each $\pi_1\inc_i$ is an epimorphism and
$$
w_1M = w_\pi : \pi \to \Z^{\times}
$$ 
or
\item $\partial M$ has one component, 
$[\pi:\mathrm{Im}(\pi_1\inc)]=2$ and
$$
w_1M = w_{\pi} \cdot w_{\partial} : \pi \to \Z^{\times}
$$
where $w_\partial: \pi \to \Z^{\times}$ has kernel $\mathrm{Im}(\pi_1\inc)$.
\qed
\end{enumerate}
\end{enumerate}
\end{theorem}

We shall also use  the equation (from the proof of Theorem \ref{cd<n})
\begin{equation} \label{PLD}
 \ol{H^k(\pi; \Z\pi)} =\ol{H^k(M; \Z\pi)} \cong \wt H_{n-k-1}(\partial \wt M),
\end{equation}
which follows from  Poincar\'e-Lefschetz duality: 
\[
\ol{H^{k}(M; \Z\pi)} \cong H_{n-k}(M, \partial M; \Z\pi),
\]
 the identification $ H_{n-k}(M,\partial M; \Z\pi)=H_{n-k}(\wt M, \partial \wt M)$, 
 and the long exact sequence of the pair $(\wt M, \partial \wt M)$.

When $n=4$ we have a realization theorem,
on the homotopy level of $PD_4$-pairs.

\begin{theorem}
\label{PD4pair}
Let $\pi$ be a finitely presentable group,  $w:\pi\to\mathbb{Z}^\times$ be a homomorphism, 
$N$ be a closed $3$-manifold and $p:N\to{K(\pi,1)}$ be a map such that $w_1N=p^*w$.
Let $X$ be the mapping cylinder of $p$ and let $\kappa=\ker(\pi_1p)$.
Suppose that one of the following conditions holds:
\begin{enumerate}
 \setcounter{enumi}{-1}
\item$\pi=1$ and $N$ is a homology $3$-sphere;
\item$\cd \pi=1$, $N$ is connected, $p_*=\pi_1p$ is an epimorphism, and $\kappa$ is perfect;
\item$\cd \pi=2$, $N$ is connected,  $p_*$ is an epimorphism, and 
$H^i(p;\Z\pi)$ is an isomorphism for $i\leq2$;
\item either
\begin{enumerate}
\item$N$ is connected, $[\pi:p_*(\pi_1N)]=2$, $\kappa$ is perfect, 
and $\pi$ is a $PD_3$ group with orientation character $w\cdot w_\partial$, 
where $w_\partial: \pi \to \mathbb{Z}^\times$ is the homomorphism with kernel $p_*(\pi_1N)$; or
\item $N = N_1 \sqcup N_2$ has two components, 
$\pi$ is a $PD_3$-group with orientation character $w$, 
each $p_{i*} : \pi_1 N_i \to \pi$ ($i = 1,2$) is an epimorphism, 
and each $\kappa_i = \ker(p_{i*})$ ($i = 1,2)$ is perfect.
\end{enumerate}
\end{enumerate}
Then the pair $(X,N)$ is a $PD_4$-pair with orientation character $w$.
\end{theorem}

\begin{proof}
We shall verify the obviously necessary conditions,
namely that $H^i(X,N;\Z\pi)=0$ if $i\not=4$ and $H^4(X,N;\Z\pi)\cong\Z$.
Hence $(\pi,N,P_*)$ is a $PD_4$-group with boundary,
as defined in \cite{dh1}, and so $(X,N)$ is a $PD_4$-pair with orientation character $w$,
by \cite[Theorem 4]{dh1}.


This is clear if $\pi=1$ and $N$ is a homology $3$-sphere, 
for then $X$ is contractible and so
$H^i(X,N)\cong\widetilde{H}^{i-1}(N)$ for all $i$.

When $\pi\not=1$ the image of $\pi_1N$ in $\pi$ is infinite,
and so $H^0(N;\Z\pi)=H^0(X;\Z\pi)=0$.
Hence $H^0(X,N;\Z\pi)=0$.

If (1) holds then $H^1(\kappa;\Z\pi)=\Hom(\kappa,\Z\pi)=0$,
since $\kappa$ is perfect. 
It then follows from the 5-term exact sequence of low degree
for $\pi_1N$ as an extension of $\pi$ by $\kappa$ that $H^1(p;\Z\pi)$
is an isomorphism.
Since $H^i(X;\Z\pi)=0$ for $i>1$ and
$H^2(N;\Z\pi)\cong H_1(N;\Z\pi) = {H_1(\kappa)}=0$, 
by Poincar\'e duality for $N$,
we see that $H^i(X,N;\Z\pi)=0$ for $1\leq{i}=3$,
and $H^4(X,N;\Z\pi)\cong{H^3(N;\Z\pi)}\cong\Z$.

If (2) holds then $H^i(X;\Z\pi)=0$ for $i>2$, 
and the exact sequence of cohomology for $(X,N)$ gives $H^j(X,N;\Z\pi)=0$ 
for $j\leq3$ and an isomorphism from $H^3(N;\Z\pi)$ to $H^4(X,N;\Z\pi)$.
Hence $H^4(X,N;\Z\pi)\cong\Z$, by Poincar\'e duality for $N$.

If (3a) holds then $H^i(X;\Z\pi)=0$ for $i<3$, 
while $H^1(N;\Z\pi)=0$ since $\kappa$ is perfect and $\pi$ has one end,
and $H^2(N;\Z\pi)\cong{H_1(N;\Z\pi)}=H_1(\kappa;\Z\pi)=0$, by Poincar\'e duality for $N$.
Hence $H^j(X,N;\Z\pi)=0$ for $j\leq2$.
Let $C=\pi/p(\pi_1N)=\Z/2$.
Then $H^3(X;\Z\pi)\cong\Z$ and $H^3(N;\Z\pi)\cong\Z{C}$, 
and $H^3(p;\Z\pi)$ is the transfer.
Hence $H^3(X,N;\Z\pi)=0$ and $H^4(X,N;\Z\pi)\cong\Z$.

The final case is similar.
\end{proof}

Parts (0) and (1) of Theorem \ref{PD4pair} are covered by Lemma 3.2 of \cite{FT05},
but we wish to point up the parallels with Theorem \ref{A}.

The following is a consequence of the previous lemma and theorem.

\begin{corollary}
\label{pi-inc inj}
Let $M$ be a compact aspherical $n$-manifold with fundamental group $\pi$.
If $\cd\pi\leq{n-2}$ and $\pi_1\inc$ is injective, then
$\partial{M}$ is connected,
$\cd\pi\leq{n-3}$ and $\pi_1\inc$ is an isomorphism.
\qed
\end{corollary}

\section{Fundamental group systems of aspherical 4-manifolds}

The work of many 3-manifold topologists, 
culminating in the Geometrization Theorem of Perelman and Thurston,
implies that compact aspherical orientable 3-manifolds are determined up to homeomorphism by their peripheral systems.
In Theorem \ref{C} we shall show that a similar result holds for 4-manifolds, 
provided that $\pi_1M$ is elementary amenable.   
In this section we examine their peripheral systems and prove Theorem \ref{A}.

A space or a group is {\em acyclic} if it has the $\Z$-homology of a point.

\begin{lemma}
\label{PD3 by perfect}
Let $N$ be a closed
 $3$-manifold such that $\nu=\pi_1N$ 
has an infinite perfect normal subgroup $\kappa$.
If $G=\nu/\kappa$ has one end, then $G$ is a $PD_3$-group.  Furthermore $H_1(\kappa)$ and $H_2(\kappa)$ both vanish.   
\end{lemma} 

\begin{proof} 
Let $N_\kappa \to N$ be  the cover with $\pi_1N_\kappa=\kappa$,
and let $C_*(N;\Z G) =C_*(N_\kappa)$ be the cellular chain complex for $N_\kappa$,
considered as a complex of finitely generated free left $\Z G$-modules.
We first claim that $N_\kappa$ is acyclic.
Note that $H_1(\kappa)=H_1(N_\kappa) = \kappa^{ab} = 0$.   
Note that 
$H_2(N_\kappa) = H_2(N;\Z G) \cong \ol{H^1(N; \Z G)} = 
\ol{H^1(G; \Z G)}  = 0$ 
where the penultimate equation holds using the cohomology spectral sequence of the homotopy fibre sequence  $N_\kappa \to N \to BG$ 
and the last equality holds since $G$ has one end.   
Hence also $H_2(\kappa)= 0$.
Next note that $H_i(N_\kappa) = 0$ for $i \geq 3$, since $N_\kappa$ is a connected open 3-manifold.

Hence $C_*(N;\Z G)$ is a finite free resolution for 
$H_0(N;\Z G)=\mathbb{Z}$ as a $\Z G$-module.  Since $H^k(G;\Z G) \cong H^k(N;\Z G) \cong H_{3-k}(N; \Z G) \cong \wt H_{3-k}(pt)$, we see that $G$ is a $PD_3$-group. 
\end{proof}

Note that $\kappa$ need not be acyclic.
Let $N=P\#\Sigma$, where $P$ is aspherical and $\Sigma\not=S^3$ a
$\Z$-homology 3-sphere, 
and let $\kappa$ be the normal closure of the image of $\pi_1\Sigma$ 
in $\nu=\pi_1N$.
Then $\kappa$ is an infinite perfect normal subgroup and 
$G=\nu/\kappa\cong\pi_1P$,  and so has one end.
It is easily seen that $H_3(\kappa)\not=0$.
However,  if $N$ is aspherical, then 
$\cd \kappa\leq2$ and $\kappa$ is acyclic.

As we have not found a convenient reference for the assertions 
about virtually polycyclic groups of Hirsch length 3 that we use in Theorem \ref{A}, 
we shall justify them here.
Our argument rests upon the fact that 
any chain of subgroups in a virtually polycyclic group is finite, 
an old result of Schur on groups with centres of finite index,
and knowledge of the finite subgroups of $GL(3,\mathbb{Z})$, 
in particular the fact that they are all polycyclic.   

\begin{lemma}
\label{NNTFN}
Let $G$ be a virtually polycyclic group with no nontrivial finite normal subgroup.
If $h=h(G)\leq3$ or if $h=4$ and $G$ is not virtually abelian,
then $G$ is polycyclic.
\end{lemma}

\begin{proof}
We may clearly assume that $h>0$, for otherwise $G$ is finite, 
and hence trivial.
Since $G$ is virtually polycyclic it has a polycyclic subgroup $P$ of finite index.
The intersection $Q$ of all subgroups of $G$ of index $[G:P]$ is a characteristic polycyclic subgroup of finite index.
The lowest nontrivial term of the derived series of $Q$
is a nontrivial characteristic abelian subgroup of $G$.
Let $A$ be a characteristic abelian subgroup which is maximal among such subgroups.
Then $A$ is finitely generated, since $G$ is virtually polycyclic, 
and so its torsion subgroup is finite.
Since the torsion subgroup is characteristic in $A$ it is normal in $G$,
and so must be trivial.
Hence $A\cong\mathbb{Z}^r$ for some $1\leq{r}\leq{h}$.
(Note also that $r\leq3$, if $h=4$, by hypothesis.) 
Our argument involves a downward induction on $h$.

The quotient $G/A$ is virtually polycyclic,
and so has an unique maximal finite normal subgroup.
Let $D$ be its preimage in $G$.
Then $D$ is characteristic in $G$ and conjugation in $D$ 
induces a homomorphism from $D$ to $\Aut(A)\cong{GL(r,\mathbb{Z})}$, 
with kernel $C$ the centralizer of $A$ in $D$.
Since $C/A$ is finite the commutator subgroup $C'$ is finite,
by Schur's Theorem \cite[10.1.4]{Ro}.
But $C$ is normal in $G$ and so $C'$ must be trivial.
Hence $C$ is abelian, and so $C=A$, by maximality of $A$. 
Therefore $D/A$ maps injectively to $GL(r,\mathbb{Z})$.
The finite subgroups of $GL(r,\mathbb{Z})$ are polycyclic, if $r\leq3$,
and so $D$ is polycyclic.

If $r=h$ then $D=G$ and we are done.
If $r<h$ then we note that $G/D$ has no nontrivial finite normal subgroup 
and $h(G/D)=h-r<h$, and so $G/D$ is polycyclic, by the inductive hypothesis.
Therefore $G$ is also polycyclic.
\end{proof}

The restriction on $h$ is necessary.
The alternating group $A_5$ acts on $\mathbb{Z}^5$ by permuting the coordinates,
and this action fixes the hyperplane $\sum {x_i}=0$. 
Hence $A_5$ acts effectively on $\mathbb{Z}^4$.
The corresponding semidirect product $\mathbb{Z}^4\rtimes{A_5}$ 
is virtually abelian of rank 4, and has no nontrivial finite normal subgroup,
but is not solvable and hence not polycyclic.

We shall now prove  Theorem \ref{A}.

\begin{proof}[Proof of Theorem \ref{A}]
The assertions about $\partial M$  follow from Lemma~\ref{components of boundary}, 
Theorem~\ref{cd<n}, and the equation ~\eqref{PLD} (from the proof of Theorem \ref{cd<n}).

We need to address the possible $\pi = \pi_1M$.

Since $M$ is an aspherical 4-manifold, $\cd\pi\leq4$, 
with equality if and only if $M$ is closed.   Since $M$ is compact and aspherical, $\pi$ has type $FP$.
Since $\pi$ is elementary amenable and $\cd\pi<\infty$,
it is virtually solvable \cite{HL92}.
A virtually solvable group which is $FP$ is constructable,
and is a duality group, with Hirsch length $h(\pi)=\cd\pi$ \cite{Kr86}.

The trivial group is the only group with cohomological dimension 0.
 
If $\cd\pi=1$ then $\pi$ is a nontrivial free group,  and so $\pi\cong\Z$,
since $\pi$ is elementary amenable.

The only finitely generated solvable groups with cohomological dimension 2 
are the Baumslag-Solitar groups $BS(1,m)$ with $m\not=0$ \cite{Gi79}.
If $\pi$ is virtually finitely generated solvable and $\cd\pi=2$ then $\pi$ has
a normal solvable subgroup $K$ of finite index;  
hence $K$ is a Baumslag-Solitar group.
Hence $\chi(\pi)=[\pi:K]^{-1}\chi(K)=0$, 
and so $\beta_1(\pi)=1+\beta_2(\pi)\geq1$.
Therefore $\Hom(\pi,\Z)=H^1(\pi;\Z)\not=0$.
The kernel of an epimorphism from $\pi$ to $\Z$ is torsion-free,
and virtually abelian of rank 1.
Hence it is abelian, and so $\pi$ is solvable,
and thus is also a Baumslag-Solitar group.   

If $\cd\pi=3$, then $\pi$ is solvable \cite{Hi24}.
Hence $H^k(\pi;\Z\pi)=0$ for $k\leq2$,
since solvable groups with type $FP$ are duality groups \cite{Kr86}.
Hence, by equation \eqref{PLD}, 
$H_1(\partial \wt M)$, $H_2(\partial \wt M)$, 
and $H_3(\partial \wt M)$ all vanish,
and so each component of $\partial \wt M$ is acyclic. 
In particular,  $\kappa_i=\ker(\pi_1\inc_i)$ is perfect for all $i$.   
The vanishing of $H_3(\partial \wt M)$ implies that each  
$G_i=\mathrm{Im}(\pi_1\inc_i)$ is infinite (in the nonorientable case 
either replace $\Z$ by $\F_2$ or pass to the 2-fold orientation double cover).   
Let $G_i=\mathrm{Im}(\pi_1\inc_i)$.
Since $\partial \wt{M}_i$ is acyclic, 
the chain complex $C_*(\partial{M}_i;\mathbb{Z}[G_i])$ 
is a finite free resolution of $\mathbb{Z}$, and so $G_i$ is a $PD_3$-group, 
as in Lemma \ref{PD3 by perfect}.

Since $\pi$ is solvable,  each $G_i$ is solvable. 
Every solvable Poincar\'e duality group $G$ is polycyclic
and $h(G)=\cd G$ \cite{Bieri},
and hence each $G_i$ is a polycyclic $PD_3$-group,
of Hirsch length 3.
It follows by an induction on the Hirsch length that if
$K\leq{H}$ are finitely generated solvable groups, $K$ is polycyclic, 
$H$ is constructable, and $h(K)=h(H)$, then $[H:K]<\infty$.
Therefore $[\pi:G_i]<\infty$, and so $\pi$ is a polycyclic $PD_3$-group.   

The final case is when $M$ is a closed aspherical manifold 
with elementary amenable fundamental group.   
This case is well-known, 
and is included here only for completeness.   
Since $\pi$ is then a virtually solvable Poincar\'e duality group,
it is virtually polycyclic, and $h(\pi)=\cd\pi=4$.
If $\pi$ is not virtually abelian, then it
is polycyclic by Lemma \ref{NNTFN}.
There are just 74 torsion-free virtually abelian groups 
of Hirsch length 4 (the flat 4-manifold groups) \cite{HiF}, 
and all are polycyclic, by inspection.  
\end{proof}

It follows immediately from Poincar\'e-Lefshetz duality for $(M,\partial{M})$
that when $\pi$ is a $PD_2$-group,   $H^i(p;\Z\pi)$ is an isomorphism
for $i\leq2$ if and only if the covering space $\partial\widetilde{M}$ 
has the integral homology of $S^1$.
(This formulation is used in part (3) of \cite[Corollary 23]{dh1}.)

The corollary below strengthens  \cite[Chapter 11.\S5]{FQ},
in which its was shown that if $M$ is an compact aspherical 4-manifold 
with $\pi=\pi_1M$ polycyclic and $f:M_1\to{M}$ is a homotopy equivalence 
which restricts to a homeomorphism on the boundaries then $f$ is 
homotopic to a homeomorphism.

\begin{corollary} \label{FJC for EA in dim 4}
If $\pi$ is the fundamental group of a compact aspherical 4-manifold and $\pi$ is elementary amenable, then the Farrell-Jones Conjecture in $K$- and $L$-theory holds.  In particular, the Borel Uniqueness Conjecture for such manifolds hold: a homotopy equivalence between compact aspherical 4-manifolds with elementary amenable fundamental group which is a homeomorphism on the boundary is homotopic, relative to the boundary, to a homeomorphism.
\end{corollary}

\begin{proof}
By Theorem \ref{A}, $\pi$ is a polycyclic or a solvable Baumslag-Solitar group as well as being finitely generated torsion-free.   The FJC for such polycyclic groups is proven in $K$-theory in \cite{FH81} and in $L$-theory in \cite{FJ88}.  The FJC  in $K$ and $L$-theory is proven for solvable Baumslag-Solitar group groups in \cite{FW14}. (See also \cite{St87} for the $L$-theory case.)

The Borel Uniqueness Conjecture is a formal consequence of the fact that $\pi$ is good and satisfies the Farrell-Jones Conjectures  (see \cite{dh1}).
\end{proof}

\begin{corollary}
\label{cd2kernel}
If $\cd \pi=2$ then $\kappa=\ker(p)$ has abelianization
$\kappa^{ab}=H_1(\kappa)\cong\overline{H^2(\pi;\Z\pi)}$, and so is nontrivial.
\end{corollary}

\begin{proof}
The exact sequence of homology and Poincar\'e duality give
\[\kappa^{ab}=H_1(\partial\widetilde{M})\cong
{H_2(\widetilde{M},\partial\widetilde{M})}
\cong\overline{H^2(M;\mathbb{Z}\pi)}=\overline{H^2(\pi;\mathbb{Z}\pi)}.
\]
Since $\cd\pi=2$ this is nonzero, and so $\kappa$ is nontrivial.\end{proof}

Thus such a group $\pi$ cannot be realizable by an aspherical manifold with $\pi_1$-injective boundary.

The polycyclic groups which are fundamental groups of closed 4-manifolds
are essentially known, 
in so far as the classification can be largely reduced to
questions of conjugacy in $GL(3,\mathbb{Z})$ (and related groups) 
See \cite[Chapter 8]{HiF}.

Elementary amenable groups of cohomological dimension 3 are solvable, 
and the finitely presentable examples are constructable and of 
 type $F$ \cite{Hi24}.
Theorem \ref{A} shows that, for example, 
$\Z \times BS(1,2)$ is not the fundamental group of a compact aspherical 4-manifold,
although it is a duality group of type $F$ and has cohomological dimension 3.
 
If we drop the condition that $\pi$ be elementary amenable then 
there are similar conclusions regarding $\partial{M}$
and $\pi_1\inc$ when $\cd \pi\not=3$.
However, when $\cd\pi=3$ we can no longer expect $\pi$ to be 
a $PD_3$-group or $\partial{M}$ to have at most two components.
For example, if $r>1$ then $\Z^2\times{F(r)}$ is not a $PD_3$-group,
but is the fundamental group of an aspherical 4-manifold with 
$k+1$ boundary components, for any $k\leq{r}$.
If $M=N \times{S^1}$ with $N = (T^2\times{I})\natural(T^2\times{I})$,
then $\partial{M}$ has 3 components and
$\pi\cong(\mathbb{Z}^2*\mathbb{Z}^2)\times\mathbb{Z}$ is not even a duality group, 
as $H^i(\pi;\mathbb{Z}\pi)$ is nonzero for $i=2,3$.

The following list of simple examples of aspherical 4-manifolds 
with nonempty boundary includes examples representing each
polycyclic group of cohomological dimension $\leq3$.
(It also includes examples whose groups have noncyclic free subgroups, 
which are not known to be good groups at the present time.)

\begin{enumerate}
 \setcounter{enumi}{-1}
\item$\pi=1$: let $M=D^4$, with $\partial{M}=S^3$;
\item$\pi=F(r)$: let $M=\natural^rS^1\times{D^3}$, with $\partial{M}=\#^rS^1\times{S^2}$;
\item$\pi=\pi_1S_g$, $S_g$ a closed surface of genus $g\geq1$: let $M$ be the total space
of a $D^2$-bundle over $S_g$, with $\partial{M}$ the associated $S^1$-bundle;
\item$\pi=\pi_1N$, $N$ an aspherical closed 3-manifold: 
let $M=M(\eta)$ be the total space of the $I$-bundle over $N$ induced by 
$\eta\in{H^1(\pi;\F_2)}$.
Then $\partial{M}=N\times\{0,1\}$ if the bundle is trivial and
$\partial{M}$ is a connected 2-fold covering space of $N$ otherwise.
\end{enumerate}

The possibilities with $\pi$ abelian are represented by the products 
$D^{4-k}\times{T^k}$, for $0\leq{k}\leq4$, 
where $T^k=\mathbb{R}^k/\mathbb{Z}^k$ is the $k$-torus.

The remaining groups allowed by Theorem \ref{A}
are the solvable Baumslag-Solitar groups.   One way to construct aspherical 4-manifolds with BS-fundamental group is as
the union of a 0-handle, two 1-handles, and a 2-handle, based on the one-relator presentation given in \S1.  
Equivalently, on can draw Kirby calculus diagrams 
with a 3-component link having two dotted components 
representing the generators,
and the third component (representing the relator) with integral framing.
Since the 2-complexes corresponding to these presentations are aspherical, 
so are the resulting 4-manifolds.

When $\pi=BS(1,1)=\mathbb{Z}^2$
we may take the link to be the Borromean rings $Bo$.
Varying the framing on the third component gives the 
different total spaces $E_n(\mathbb{Z}^2,1)$  
of orientable $D^2$-bundles over the torus.

Note that  every group $\pi$ listed in Theorem \ref{A} is the fundamental group of a {\em smooth} compact aspherical 4-manifold.

The first two examples in the above list are essentially the only ones 
with $\cd\pi<3$ and $\pi_1$-injective boundary.
For then $\partial{M}$ is connected and $\pi_1\inc$ is an isomorphism, 
by Corollary \ref{pi-inc inj}.
Hence $\pi$ is free of finite rank $r\geq0$, 
since $\pi\cong\pi_1\partial{M}$ and $\cd\pi<3$.
(More generally, if $M$ is a compact 4-manifold with connected, nonempty boundary and $\pi_1\inc$ is an isomorphism then $\pi$ must be free
\cite{Da94}.)
Hence $M$ is orientable if and only if $\partial{M}$ is orientable, 
and $(M,\partial{M})\simeq\natural^r(S^1\times{D^3},S^1\times{S^2})$ 
or $\natural^r(S^1\tilde\times{D^3},S^1\tilde\times{S^2})$,
accordingly. 

The examples with $\cd\pi=3$ include all the possibilities 
with $\pi$ a $PD_3$-group and all the $\pi_1\inc_i$ injective.
For then $\pi_1\partial_i{M}$ is also a $PD_3$-group,
since it has finite index in $\pi$, by Theorem \ref{cd<n}.
Hence $\partial_i{M}$ is aspherical.
It then follows from Mostow Rigidity and the Geometrization Theorem 
(via \cite{MS86,Zi82}) that $\pi$ is also a 3-manifold group.

We shall examine the role of the boundary and peripheral system 
in determining $M$ in \S6 below.

\section{existence: realization of boundaries}

It is an immediate consequence of Theorem \ref{PD4pair} that every closed 3-manifold $N$, 
homomorphism from $\pi_1N$ to $\pi$,
and orientation character $w:\pi\to\mathbb{Z}^\times$
compatible with Theorem \ref{A} can be realized by an aspherical $PD_4$-pair $(X,N)$,
since homomorphisms  from $\pi_1N$ to $\pi$ correspond to maps from $N$ to $K(\pi,1)$.
Theorem B then follows from this result and the Borel Existence Theorem for pairs,
which we state here (as specialised to the 4-dimensional,  $\pi$ elementary amenable
case) for convenience:

\smallskip
\noindent
{\bf Borel Existence Theorem for pairs}
\cite[Theorem C]{dh1}.
{\it Let $(X,Y)$ be a Poincar\'e pair of dimension $4$ with $X$ aspherical 
with fundamental group $\pi$ and with $Y$ a nonempty closed $3$-manifold.
If $\pi$ is elementary amenable then there is a $4$-manifold $M$ and 
a homotopy equivalence $f:(M,\partial{M})\to(X,Y)$ which restricts to a homeomorphism of boundaries.}
 
\smallskip
We shall comment briefly on simple examples.

The case when $\pi=1$ is well understood.
The manifold $M$ must be contractible, 
and taking the boundary gives a bijective correspondence between compact
contractible topological 4-manifolds and homology 3-spheres 
\cite[Corollary 9.3C and Proposition 11.6A]{FQ}. 
The simplest nontrivial smooth examples of this type are the Mazur 4-manifolds, 
formed by adding a 2-handle to $D^3\times{S^1}$ with attaching circle homologous 
to the $S^1$ factor on the boundary.

More generally, if $M$ is a 4-manifold with nonempty boundary then
taking boundary connected sum with a contractible 4-manifold does not change $\pi$ 
(or the homotopy type of $M$)
 but changes $\partial{M}$ by connected sum with a homology 3-sphere.
(Conversely, every such $M$ with $\pi$ elementary amenable and 
$\partial{M}\not=\emptyset$ is a boundary connected sum $M_0\natural{C}$, 
where $C$ is contractible and $\partial{M_0}$ is prime,
by Corollary \ref{prime bdry} below.)

When $\pi\cong\Z$ the realization is essentially due to 
Freedman and Quinn  \cite[Proposition 11.6A]{FQ}.
The simplest nontrivial smooth example of this type is perhaps the 
exterior $X(\Delta)$ of the slice disc for the Kinoshita-Terasaka 11-crossing knot 
$11_{n42}$ deriving from a standard ribbon disc, as in \cite[Figure 1.4]{HiA}.
The exterior  is homotopy equivalent to $S^1$, 
but $\partial{X(\Delta)}$ is aspherical, and so is not $S^2\times{S^1}$.

We may construct nontrivial examples of 4-manifolds with $\pi\cong\Z^2$ as follows.
The Kirby diagram given by the Borromean rings $Bo$
with two dotted components and one 0-framed component 
is a presentation for $(D^2\times{T^2},T^3)$.
If we tie Alexander polynomial-1 knots in one or more of the components
of $Bo$ we obtain a Kirby diagram for a 4-manifold $M\simeq{T^2}$,
and with $\partial{M}$ having the $\mathbb{Z}[\mathbb{Z}^2]$-homology of  $S^1\times{T^2}$. 
(Other integer framings of the third component give 4-manifolds with boundaries 
$\mathbb{Z}[\mathbb{Z}^2]$-homology equivalent to $S^1$-bundles over $T^2$.)

We shall give a related construction of $\mathbb{Z}[\mathbb{Z}^3]$-homology $T^3$s in Example  \ref{4d_he_with_homeo_b} below.

\section{uniqueness: the role of the peripheral system} \label{section:uniqueness}

In this section we shall prove Theorem \ref{C}, 
which asserts that an enhanced form of the peripheral system is 
a complete invariant for the homeomorphism type of 
a compact aspherical 4-manifold with elementary amenable fundamental group and orientable boundary.  We also study the geometric topological properties of the boundary; the algebraic topological properties of the boundary were completely determined by Theorems \ref{A} and \ref{B}.

We begin by defining the notion of enhanced peripheral system.
If $M$ has orientable boundary and $N$ is a component of $\partial{M}$,
let $[N]$ be the image of a fundamental class $[M]$
in $H_3(\pi_1N;\Z)$ under the composition of the connecting homomorphism
from $H_4(M, \partial M;\Z^w)\to H_3(\partial M;\Z)$, the projection onto the summand
$H_3(N;\Z)$, and the homomorphism to $H_3(\pi_1N;\Z)$ induced by the classifying map.
The {\it enhanced\/} peripheral system of a compact 4-manifold $M$ 
with orientable boundary is the peripheral system together with the  classes $[N]$.
Two such enhanced peripheral systems are equivalent if there is an isomorphism of peripheral systems which preserves each of the homology classes $[N]$, 
up to a simultaneous change of signs.
(The enhancement is used to identify $\partial{M}$ among all 3-manifolds
with given fundamental group.
It is redundant if $\partial{M}$ is a prime 3-manifold;
in particular, if $\pi$ is elementary amenable and the boundary is $\pi_1$-injective.)
We shall give an example after the proof of Theorem \ref{C} to 
show that enhancement is necessary in general.

The infinite dihedral group $D_\infty = \Z \rtimes_{-1} \Z/2$ is isomorphic to the free product $\Z/2 * \Z/2$.

\begin{lemma}
\label{free product}
Let $\nu$ be a group which has a perfect normal subgroup $\kappa$
and solvable quotient $\nu/\kappa$.
If $\nu ={G*H}$, then either $G$ or $H$ is perfect, or $\nu/\kappa$ is the infinite dihedral group.

\end{lemma}

\begin{proof} Assume that $\kappa$ is a perfect normal subgroup of $\nu = G * H$,  that $G$ and $H$ are not perfect, and that $\nu/\kappa$ is solvable.   We want to show that $\nu/\kappa$ is infinite dihedral.

We are going to use the fact \cite[Proposition 1.4.6]{DD89} that  for any groups $A$ and $B$,  $\ker(A * B \to A \times B)$ is a  free group.   If $A$ and $B$ have order 2  the kernel has rank 1, but if $A$ and $B$ are both nontrivial and $A$ has order greater than 2, then the kernel is nonabelian ($ab^{-1}a^{-1}b^{-1}$ and $a'b^{-1}a'^{-1}b^{-1}$ don't commute if $a,a'$ are distinct and not  1 and $b$ is not  1), hence has rank greater than one.

The length of the derived series of $\nu$ is finite by hypothesis, thus so is the length of any quotient $\rho$ of $\nu$.   Let $\ol G$ and $\ol H$ be nontrivial solvable quotients of $G$ and $H$.    Since the length of the derived series of  $\rho = \ol G *\ol H$ is finite, the rank of the free group  $\ker(\ol G * \ol H \to \ol G \times \ol H)$ must be 1, which implies that  $|\ol G| =2$ and $|\ol H| = 2$.  Applying this fact when $\ol G = G/[G,G]$ and $\ol H = H/[H,H]$ (which are nontrivial due to the assumption that $G$ and $H$ are not perfect), we see that the commutator subgroups have index 2, and then applying this fact again when $\ol G = G/[[G,G],[G,G]]$ and  $\ol H = H/[[H,H],[H,H]]$ we see that the commutator subgroups are perfect and index 2.

Let $P = [G,G]$ and $Q = [H,H]$ be the perfect index two subgroups of $G$ and $H$ respectively.   Let  $\varphi : G * H \to G/P *H/Q \cong D^\infty$ be the free product of the epimorphisms.   The kernel of $\varphi$ is generated by conjugates of elements of $P$ and of elements of $Q$, and hence is perfect.   Since $\nu/\kappa$  is solvable,  $\ker \varphi \subset  \kappa$.  But $\kappa = [[\kappa,\kappa],[\kappa,\kappa]] \subset \ker \varphi$, since $[[D^\infty,D^\infty],[D^\infty,D^\infty]] = 1$.  Thus there is an induced isomorphism $\nu/\kappa \xrightarrow{\cong} D^\infty$.
\end{proof}

We need some results from 3-manifold topology.  
A connected closed 3-manifold is {\em prime} if it cannot be expressed as a nontrivial connected sum.   
A prime manifold is either aspherical, or contains a 2-sided projective plane, 
is diffeomorphic to $S^1 \times S^2$ or $S^1 ~\tilde \times~ S^2$, 
or has finite fundamental group. 
If a 3-manifold $N$ contains a 2-sided projective plane $\mathbb{RP}^2$ then the orientation character
$w_1N$ is nontrivial on the image of $\pi_1\mathbb{RP}^2$, and thus $\pi_1N  $ has 2-torsion.
Every connected closed 3-manifold $N$ has a decomposition as a finite connected sum 
of prime 3-manifolds.
If $N\not\cong{S^3}$ then we may assume that none of the summands are $S^3$,
and if moreover $N$ is orientable the decomposition is then essentially unique.
If $N$ is a connected closed 3-manifold with $\pi_1N = G_1 * G_2$, 
then $N = N_1 \# N_2$ with $\pi_1N_i = G_i$.    
We also need two consequences of Perelman's Geometrization Theorem: 
the Poincar\'e Conjecture and  the fact that the Poincar\'e homology sphere $S^3/I^*$ 
is the only homology 3-sphere with nontrivial finite fundamental group.

\begin{lemma}
\label{ZHS summand}
Let $M$ be a compact aspherical $4$-manifold, and let $N$ be a component of $\partial{M}$.
\begin{enumerate}
\item $N$ has no $2$-sided projective planes.
\item Any summand of $N$ with finite fundamental group is homeomorphic to $S^3$ or $S^3/I^*$.
\end{enumerate}
Assume now that  $\pi=\pi_1M$ is elementary amenable.
 \begin{enumerate}
\addtocounter{enumi}{2}
\item 
If $\nu=\pi_1N = G * H$, then $G$ or $H$ is perfect.
\item $N$ is a connected sum $N_0\#\Sigma_N$, 
where $N_0$ is a prime manifold and is either aspherical or is $S^1 \times S^2$ 
or is $S^1~ \tilde \times ~S^2$ or is $S^3$
and $\Sigma_N$ is a homology $3$-sphere.
\end{enumerate}
\end{lemma}

\begin{proof} 
(1)
Suppose that $N$ contains a 2-sided projective plane $P$.
Then $P$ would have a product neighbourhood in $M$ and so $w_1P=w|_P$,
where $w=w_1M$.
But then $\pi_1P$ would map injectively to $\pi$, 
contradicting the fact that $\pi$ is torsion-free.

(2)
Suppose that $N = N_1 \# N_2$  and that $\pi_1N_1$ is finite.  
Thus $\nu=\pi_1N\cong\pi_1N_1 * \pi_1N_2$.  
There is an exact sequence
\begin{equation*}
1 \to \kappa \to \nu \xrightarrow{\pi_1(\inc)} \pi
\end{equation*}
where $\kappa := \ker (\pi_1(\inc))$.
We claim that $H_1\kappa = \kappa^{ab}$ is $\Z$-torsion-free.   Indeed
\begin{equation*}
H_1\kappa = H_1(N;\Z\pi) \subset H_1(\partial M; \Z\pi) \xleftarrow{\cong} H_2(M,\partial M; \Z\pi) \cong \ol{H^2(\pi;\Z\pi)}
\end{equation*}
is torsion-free, since the second cohomology group is torsion-free for any finitely presented group (see  \cite[Proposition 13.7.1]{Ge}). 

Since $\pi$ is torsion-free,  $\pi_1N_1 \subset \kappa $.
Thus $H_1N_1 \to H_1\kappa \to H_1N$ is the zero map, 
since the left group is finite and the middle group is torsion-free.   
But $H_1N = H_1N_1 \oplus H_1N_2$.   Thus $H_1N_1 = 0$.

Assume now that $\pi$ is elementary amenable.
(3)  If $\cd \pi = 0, 1,$ or $3$, then $H_1\kappa \cong  \ol{H^2(\pi;\Z\pi)} = 0$, so $\kappa$ is perfect.   Theorem \ref{A} shows that $\pi$ is solvable, and $\pi$ is torsion-free since $M$ is finite-dimensional and aspherical.   The result then follows from Lemma \ref{free product}.

When $\cd\pi=2$,  we must extend our strategy.  
Recall that $N = \partial M$ (Lemma \ref{cdpin-2}) and that $\pi_1(\inc)$ is onto (Theorem \ref{cd<n}). 
Note that $\pi = BS(1,m) = \Z[1/m] \rtimes \Z$ has a composition series with two torsion-free abelian factors.
Since $\nu/[\kappa,\kappa]$ is an extension of $\pi$ by a torsion-free abelian group  $\kappa^{ab} \cong   \ol{H^2(\pi;\Z\pi)}$, it follows that 
 $\nu/[\kappa,\kappa]$ has a composition series with three torsion-free abelian factors.
Thus any homomorphism from a group with finite abelianization to
$\nu$ must have image in $[\kappa,\kappa]\leq[\nu,\nu]$. 
It follows that if $\nu = {G*H}$ and $\beta_1(H)=0$, then $H$ is perfect.
(Here $\beta_1(H)=\rank H_1(H)$.)

Suppose first that $M$ is orientable.
If $m\not=1$, then $H_1(\pi)\cong \Z$ and $H^2(\pi)\cong\mathbb{Z}/(m-1)$,
so $\beta_1(\nu)=1$, by the long exact sequence for $(M,N)$.
Thus if $\nu = {G}*{H}$ with $\beta_1(G)\geq\beta_1(H)$ then $\beta_1(H)=0$, and so $H$ is perfect.
If $m=1$, so $\pi\cong\mathbb{Z}^2$, then $\beta_1(\nu)=2$ or 3,
and $H^2(\pi;\Z\pi)\cong\mathbb{Z}$.
If $\nu = {G}*{H}$ with $\beta_1(G)=\beta_1(H)=1$ then the epimorphism from $\nu$ to $\pi$ factors through an epimorphism to $F(2)$.
But then $H_1(\nu;\Z\pi)$ maps onto $H_1(F(2);\Z\pi)$,
which has infinite rank as an abelian group.
Thus if $\beta_1(\nu)=2$ and $\beta_1(G)\geq\beta_1(H)$ then 
$H$ must be perfect.
A similar argument applies if $\beta_1(\nu)=3$.

If $M$ is nonorientable and $\nu=\pi_1N = {G*H}$ 
with $\beta_1(G)\geq \beta_1(H)>0$, then the orientable double cover $M^+$
has $\nu^+=\pi_1\partial{M^+} = {G_1*H_1}$ with $\beta_1(G_1),\beta_1(H_1)>0$ (use the Kurosh's subgroup theorem and a transfer argument).
But we have shown that this never occurs in the orientable case.   Thus $\beta_1(H) = 0$ and $H$ is perfect.

(4) If $N$ is a homology sphere we are done by (3).   Otherwise 
we may express $N = N_0 \# N_1$ where $N_0$ is a prime manifold which is not an homology sphere.    Then $N_1$ is a homology sphere by part (3).
By part (2), the prime manifold $N_0$ does not have finite fundamental group, 
and so it must be aspherical or $S^1 \times S^2$ or  $S^1~ \tilde \times ~S^2$.
\end{proof}


It follows immediately from Theorem \ref{cd<n} that $w_1M$ is determined by $w_1\partial{M}$ if $\cd\pi\leq2$, or by $w_1\pi$ and the homomorphism
$\pi_1\partial{M}\to\pi_1M$,  if $\pi$ is a $PD_3$-group.
If $N$ is a boundary component of $M$ then there is a natural isomorphism 
$H^1(N;\mathbb{F}_2)=H^1(N_0;\mathbb{F}_2)$, 
since $\Sigma_N$ is a homology 3-sphere,
and $w_1N=w_1N_0$.
If $N_0$ is aspherical or $S^3$ then $w_1N_0$ is determined by 
$\pi_1N_0$.
Hence the orientation characters are determined by the other data of the peripheral system {\it except\/} when $\pi=\mathbb{Z}$ and 
$\pi_1N=\mathbb{Z}*P$, with $P$ perfect.
There are then {\it two\/} possible orientation characters.
These are realized by $M=(D^3\times{S^1})\natural{C}$ or
 $(D^3~\tilde\times~{S^1})\natural{C}$, where $C$ is a compact contractible 4-manifold with $\pi_1\partial{C}=P$.
 
We shall now prove Theorem \ref{C}.

\begin{proof} It is clear that if $M$ and $\wh M$ are homeomorphic then their enhanced peripheral systems are equivalent.

Assume now that $M$ and $\wh M$  are compact aspherical 4-manifolds with elementary amenable fundamental groups and orientable boundaries $N$ and $\wh N$ and that their enhanced peripheral systems are equivalent.   
Corollary \ref{FJC for EA in dim 4}
says that the Borel Uniqueness Conjecture holds for $M$ and $\wh M$, so it suffices to show there is a homeomorphism $h : N \to \wh N$ which extends to a homotopy equivalence $M \to \wh M$.   This is automatic if the boundary is empty by asphericity, so we assume that $\cd \pi <4$.

We assume first that $\cd \pi <3$, in which case $N$ and $\wh N$ are
connected by Lemma \ref{components of boundary} and  $\pi_1\! \inc : \pi_1 N \to \pi_1 M$ and   $\pi_1 \wh \inc : \pi_1 \wh N \to \pi_1 \wh M$ are epimorphisms by Theorem \ref{cd<n}.  The   peripheral system assumption is then that there are isomorphisms 
$\theta : \pi_1 M \xrightarrow{\cong} \pi_1 \wh M$ and $\varphi : \pi_1 N \xrightarrow{\cong} \pi_1 \wh N$ so that 
$\theta \circ  \pi_1\!\inc=\pi_1\wh \inc  \circ \varphi$.  (Since $\pi_1\!\inc$ and $\pi_1\wh \inc$ are epimorphisms, 
the conjugation in the definition of peripheral system can be absorbed into 
the isomorphism $\theta$.)  The enhanced assumption is that there are orientations $[M]$ and  $[M']$ so that $\theta_*[M] = [M']$ 
and $\varphi_*(H_3(c)[N]) = H_3(\wh c)[\wh N] \in H_3(\pi_1 \wh N)$ for the induced  orientation on the boundaries, 
where  $c : N \to K(\pi_1N,1)$ and $\wh c : \wh N \to K(\pi_1\wh N,1)$ are maps inducing the identity on the fundamental groups.

Since $N$ and $\wh N$ are connected and $\pi_1\!\inc$ and $\pi_1\wh\inc$ are epimorphisms,  elementary obstruction theory (see Proposition 9 of \cite{dh1}) and the Borel Uniqueness Conjecture says that it suffices to find a homeomorphism $h :N \xrightarrow{\cong} \ \wh N$ and an isomorphism $\theta : \pi_1 M \xrightarrow{\cong} \pi_1 \wh M$  so that $\theta \circ  \pi_1\!\inc=\pi_1\wh \inc  \circ \pi_1 h$.

So our task is to construct this homeomorphism $h$.  The key fact is that an isomorphism between
the fundamental groups of two closed orientable prime 3-manifolds which are not lens spaces can be realized by a homeomorphism \cite[Theorem 1.4.3]{BBBMP}.   This is a culmination of 
deep work in 3-manifold topology, with the last step being the geometrization theorem of Perelman.

  But lens spaces are not prime summands of $N$ and $\wh N$ by part (2) of Lemma \ref{ZHS summand}.   Thus the prime summands of $N$ and $\wh N$ are homeomorphic. But the construction of  $N$ and $\wh{N}$ from their summands involves 
choices of discs to form the connected sum, 
and it is at this point that we need to keep track of orientations.   This is where the enhancement is needed.  

Choose prime decompositions   $N=\#_{i=0}^rN_i$ and $\wh{N}=\#_{j=0}^r\wh{N}_j$  with  $$\varphi =  \varphi_0 * \cdots *  \varphi_r  : \pi_1N_0  * \cdots *  \pi_1N_r \xrightarrow{\cong}  \pi_1\wh N_0  * \cdots *  \pi_1 \wh N_r.$$
Note that all summands are oriented.   We now choose homeomorphism $h_i : N_i \to \wh N_i$ which induce $\varphi_i$.   We would like these homeomorphisms to be orientation preserving.   If $N_i$ is aspherical or the Poincar\'e homology sphere, then the classifying map  induces an isomorphism $H_3(N_i) \xrightarrow{\cong} H_3(\pi_1N_i)$, so the enhancement assumption guarantees that $h_i$ is orientation preserving.   Likewise, we may assume that $h_i$ preserves orientation on the $S^2 \times S^1$ summands, by composing, if necessary, with a orientation reversing self homeomorphism which induces the identity on the fundamental group.   Then $h = h_0 \# \cdots \# h_r$ is our desired homeomorphism.

A similar argument applies to each component if $\pi$ is a $PD_3$-group 
and $\partial{M}$ and $\partial\widehat{M}$ each have two components.
In this case the homomorphisms $\pi_1\inc_i$ are epimorphisms.

Finally, if $\pi$ is a $PD_3$-group and $N$ and $\wh{N}$ are connected
then $M$ is nonorientable, and $\pi_1\!\inc$ and $\pi_1\wh\inc$  have images
$\ker w_1M$ and $\ker w_1\wh{M}$, respectively.
A similar argument applies also in this situation.
\end{proof}

The necessity of some level of enhancement is clear from the following example.
Let $C$ be a contractible 4-manifold with boundary
$P=S^3/I^*$, and fix an orientation for $C$.
Since self-homotopy equivalences of $P$ are orientation-preserving,
$P\#{P}$ is not homeomorphic to $P\#-\!P$.
Hence the two boundary connected sums $M_+=C\natural{C}$ and 
$M_-=C\natural-\!C$ are not homeomorphic, 
although their peripheral systems are both $\{I^**I^*\to1\}$.
In this case the enhanced peripheral systems are not equivalent.

The hypothesis that $\pi$ is elementary amenable is not need to show that
$\partial{M}$ and $\partial{\wh{M}}$ are homeomorphic.
If we assume only that the boundaries are orientable and the Farrell-Jones Conjectures hold for 
$\pi$ then we may show that $M\times{S^1}\cong\wh{M}\times{S^1}$,
and hence that $M$ and $M'$ are $s$-cobordant $rel~\partial$.
The main obstruction to extending Theorem C to the nonorientable case is that 
the Borel Uniqueness Conjecture is not yet known for closed nonorientable aspherical 3-manifolds.
(The Geometrization Theorem of Perelman and Thurston has not yet been proven for 
nonorientable  3-manifolds.)

\begin{corollary}
\label{prime bdry}
Let $M$ be an compact aspherical orientable $4$-manifold with 
elementary amenable fundamental group.
Then $M\cong{M_1\natural{C}}$,
where $M_1\simeq{M}$,
$\partial{M}_1$ is prime with infinite fundamental group and $C$ is contractible.
\end{corollary}

\begin{proof}
We may assume that $\partial{M}=N\#\Sigma$, 
where $N$ is prime and $\Sigma$ is a homology 3-sphere, 
by Lemma \ref{ZHS summand}.
Thus $\partial{M}=N_o\cup\Sigma_o$, 
where $N_o$ and $\Sigma_o$ are the complements of open 3-discs in 
$N$ and $\Sigma$, respectively.
The homology 3-sphere $\Sigma$ bounds a contractible 4-manifold $C$.
Then $\partial{C}=\Sigma_o\cup{D^3}$.
Let $M_1=M\cup_{\Sigma_o}C$. 
The inclusion of $M$ into $M_1$ induces an isomorphism $\pi_1M\cong\pi_1M_1$,
since $\pi_1\Sigma_o=\pi_1\Sigma$ has trivial image in $\pi$.
It also induces isomorphisms $H_i(M;\Z\pi)\cong H_i(M_1;\Z\pi)$ for all $i$,
since $H_i(\Sigma_o;\Z\pi)=0$ for $i>0$.
Hence this inclusion is a homotopy equivalence $M\simeq{M_1}$.

An orientation for $M$ determines an orientation for $\partial{M}$
and hence orientations for $N$,$\Sigma$, $C$ and $M_1$.
There is a well-defined boundary connected sum $M_1\natural{C}$
compatible with these orientations.
Let $f:N_o\to\partial{M}$  and  $h:M\to{M_1}\to{M_1\natural{C}}$ be the 
obvious inclusions, and let $g:\partial{M}\to\partial{M_1\natural{C}}$ 
be the obvious homeomorphism (the identity!).
Then $h\circ\inc\circ{f}$ is the natural inclusion of $N_o$ into
$\partial{M_1\natural{C}}$.
The homomorphisms $\varphi=\pi_1g$ and $\theta=\pi_1h\circ \pi_1 \inc$ 
define an equivalence of enhanced peripheral systems.
Hence $M$ is homeomorphic to $M_1\natural{C}$, by the theorem.
\end{proof}

\section{Classification} 
Theorems B and C suggest a classification of compact aspherical 4-manifolds $M$
with elementary amenable fundamental group $\pi$.
We shall assume throughout this section that $\partial{M}$ is nonempty, 
as the absolute case is part of  \cite[Theorem 11.5]{FQ}.

If $\pi=1$ or $\Z$ then $M$ is determined by $\pi$ and $N=\partial{M}$ alone, 
but further invariants are needed in general.
A classical example where a compact aspherical manifold is not determined by its group and its boundary is given by the exteriors of the granny knot and the square knot. 
We shall give examples to show that this can also occur for 4-manifolds 
with elementary amenable fundamental group. 

\begin{lemma}
\label{epiaut}
Let $\alpha,\beta:\nu\to\pi$ be epimorphisms with the same kernel $\kappa$.
Then there is an automorphism $\theta$ of $\pi$ such that $\theta\circ\alpha=\beta$.
\end{lemma}

\begin{proof}
The epimorphisms $\alpha$ and $\beta$ induce isomorphisms
$\overline\alpha$ and $\overline\beta$ from $\nu/\kappa$ to $\pi$.
Then $\theta=\overline\beta\circ\overline\alpha^{-1}$ is an automorphism of $\pi$ such that $\theta\circ\alpha=\beta$.
\end{proof}

If $\nu = \pi_1(N,*)$, the groups $\Homeo(N,*)$ and $\Aut(\pi)$ act on the set $\Epi(\nu,\pi)$ 
of epimorphisms from $\nu$ to $\pi$ by pre- and post-composition.
If $N$ is orientable then $\Homeo(N,*)$ acts through the subgroup of $\Aut(\nu)$ which preserves the image of $\pm[N]$ in $H_3(\nu)$.
(Note also that as the 3-manifolds $N$ arising here have no lens space
summands and are determined up to homeomorphism  by $\nu=\pi_1N$
and the image of $\pm[N]$ in $H_3(\nu)$, if $N$ is orientable,
and up to homotopy equivalence by $\nu$ alone,
if $N$ is nonorientable.)

If $G$ is a group let $I(G)$ be the preimage in $G$ of the torsion subgroup of $G^{ab}$.

Our final theorem is Theorem  D,  
in a more precise formulation with includes also the
cases with $\pi$ a polycyclic $PD_3$-group and the boundary $N$ having two components.

\begin{theorem}
\label{thmD}
Let $\pi$ be a finitely presented elementary amenable group and $N$ a closed $3$-manifold.
Assume that $p:N\to{K(\pi,1)}$ is a map such that $\pi$, $N$ and $p$ 
satisfy one of the conditions of Theorem \ref{A} (with $p$ for $inc$ in case $(2)$).
Then 
\begin{enumerate}
\item{} compact aspherical $4$-manifolds $M$ with $\pi_1M\cong\pi$, $\partial{M}\cong{N}$
and $w_1M=w$ are classified by $\pi$ and $N$ alone if 
$\pi=1$ or $\Z$.
\end{enumerate}
In the other cases further invariants are needed:

\begin{enumerate}
\addtocounter{enumi}{1}
\item{}If $\pi\cong{BS(1,m)}$, $\nu=\pi_1N$, and
\begin{enumerate}
\item($m=\pm1$) then  the quotient of $\Epi(\nu,\pi)$  by the actions of $\Homeo(N,*)$ and $\Aut(\pi)$is trivial;
\item($|m|>1$) then there are at most two possibilities, corresponding to certain elements of the quotient of $\Epi(\nu,\pi)$ by the action of $\Homeo(N,*)$.  
\end{enumerate}
\item{If} $\pi$ is a polycyclic $PD_3$-group and
\begin{enumerate}
\item($N=N_1\sqcup{N_2}$, with $\nu_1=\pi_1N_1$ and $\nu_2=\pi_1N_2$)
then also an element of the quotient of 
$\Epi(\nu_1,\pi)\times{\Epi(\nu_2,\pi)}$ by the action of $\Homeo(N_1)\times{\Homeo(N_2)}$ and the diagonal action of $\Aut(\pi)$; 
\item($N$ connected, with $\nu=\pi_1N$) then also an element of the quotient of\\
 $\bigsqcup_{\{w\}}\Epi(\nu,\ker(w))$ 
by the actions of $\Homeo(N,*)$ and $\Aut(\pi)$, where the union is taken over all  epimorphisms $w:\pi\to\Z^\times$.
\end{enumerate}
\end{enumerate}
Each set of invariants satisfying the conditions of Theorem \ref{A} can be realized.
\end{theorem}

\begin{proof}
In cases (1) and (2), Lemma \ref{components of boundary} and Theorem \ref{cd<n}, show that $N$ is connected and $
\pi_1\inc$ is an epimorphism.   It then follows from Corollary \ref{FJC for EA in dim 4} that $M$ is determined by $\pi$, $N$ and the image of $\pi_1\inc$ in the quotient
of $\Epi(\pi_1N,\pi)$  by the actions of $\Homeo(N,*)$ and $\Aut(\pi)$.

If $\pi = 1$, then $\Epi(\pi_1N,\pi)$ is trivial, and if $\pi = \Z$, then Theorem \ref{A} shows that the quotient of $\Epi(\pi_1N,\pi)$ by $\Aut(\pi)$ is trivial.


We next need to explain the more explicit result in case (2b).
Suppose that $\pi=\pi_1M\cong{BS(1,m)}=\langle{a,t}\mid{tat^{-1}=a^m}\rangle$,  where $|m|>1$.
Then the generator $t$ determines an isomorphism $\pi/I(\pi)\cong\Z$,
and hence an isomorphism $\Z[\pi/I(\pi)]\cong\Lambda=\Z[t,t^{-1}]$.

Since $[\pi,\pi]$ is torsion-free and abelian, 
every homomorphism from $\nu$ to $\pi$ must factor through $\nu/I([\nu,\nu])$. 
The exact sequence of $(M,N)$ and Poincar\'e duality together imply that $\beta_1(N)=1$,
since $H_2(M,N)\cong{H^2(M;\Z^w)}$ is finite.
Hence $\nu/I(\nu)\cong\pi/I(\pi)$.
We assume now that we have fixed such an isomorphism.
The subgroup $A=I(\pi)\cong\Lambda/(t-m)\Lambda$ is torsion-free and of rank 1 
as an abelian group.
The exact sequence of $(M,N)$ with coefficients in $\Lambda=\Z[\pi/I(\pi)]$ 
gives a short exact sequence
\[
0\to{H_2(M,N;\Lambda)}\to{I(\nu)^{ab}}\to{A=I(\pi)}\to0,
\]
since $H_2(M;\Lambda)=H_2(A;\mathbb{Z})=A\wedge{A}=0$.
Now 
\[
H_2(M,N;\Lambda)\cong{H^2(M;{}^w\Lambda)}\cong
{\Ext^1_\Lambda(A,{}^w\Lambda)}\cong\Lambda/(mt-w(t))\Lambda,
\]
by Poincar\'e duality and the Universal Coefficient spectral sequence,
and so $I(\nu)^{ab}$ is torsion-free and of rank 2 as an abelian group.
Let $E$ be the preimage of $H_2(M,N;\Lambda)$ in $\nu$.
Then $E$ is characteristic in $\nu$, since
its image in $I(\nu)^{ab}$ is the kernel of multiplication by $mt-w(t)$.

If we use coefficients $\Z[\pi^{ab}]$ instead of $\Lambda$,
a similar argument shows that $[\nu,\nu]^{ab}$ has rank 2.
Hence the quotient of $[I(\nu),I(\nu)]$ by $I(\nu,\nu)$ is a torsion group.
Since $\pi$ is torsion-free and metabelian, the image of this quotient in $\pi$ is trivial.
Hence every epimorphism from $\nu$ to $\pi$ factors through $\nu/[I(\nu),I(\nu)]$.
In all cases, such epimorphisms have kernel $E$ and so induce isomorphisms 
$\nu/E\cong\pi$.
If $(M',N)$ is another such pair and the epimorphisms $p=\pi_1\inc$ 
and $p'=\pi_1\inc'$
induce the same isomorphism from $\nu/I(\nu)$ to $\pi/I(\pi)=\Z$ 
then $\ker(p)=\ker(p')$.
Hence we may again invoke Lemma \ref{epiaut}.

The effect of changing our choice of identification of $\nu/I(\nu)$ with $\pi/I(\pi)$
is to invert the action of $t$ on $I(\nu)^{ab}$, and thus to replace
$E$ by the preimage in $\nu$ of the kernel of multiplication by $t-w(t)m$.
Thus for such $\pi$ and $N$ there are two possible kernels.
(These may be equivalent under composition with automorphisms of $\nu$ 
induced by self-homeomorphisms of $N$.)


In case 3(a) the group $\pi$ is a canonical quotient of each of $\nu_1$
and $\nu_2$, since $\pi$ is solvable and the kernels are perfect.
Hence the actions of $\Homeo(N_1,*)$ and $\Homeo(N_2,*)$ induce homomorphisms 
from these groups to $\Aut(\pi)$. 
If one of these homomorphisms is onto then $\Homeo(N_1,*)\times{\Homeo(N_2,*)}$ and   $\Aut(\pi)$ together act transitively on $\Epi(\nu_1,\pi)\times{\Epi(\nu_2,\pi)}$,
and the classification needs only $\pi$ and $N$.  
\end{proof}

When $\pi\cong{BS(1,2)}$ we recover some of the results of \cite{CP}, 
who show that if a knot $K$ bounds a homotopy ribbon 2-disc $D_K\subset{D^4}$ 
with exterior $X(D_K)$ such that $\pi_1(X(D_K))\cong{BS(1,2)}$ then 
$X(D_K)$ is aspherical,  and $K$ has at most 2 topologically distinct such discs.
They give examples of knots for which the disc is essentially unique,
and of knots with 2 such discs.

We now give examples of nonhomeomorphic compact aspherical 4-manifolds with the same boundary and fundamental group.  The examples 
 are  based on the same 3-manifold,
and rely on the canonical nature of the geometric decomposition of a 3-manifold.

\begin{example} \label{4d_he_with_homeo_b}
If $L\subset{S^3}$ is a knot or link then $X(L)$ is the complement 
of the interior of a regular neighbourhood of $L$.
If we fix orientations for $S^1$ and $S^3$ then 
the longitudes and meridians of a link 
determine `coordinates" for $\partial{X}(L)$ and
a basis for $H_1(\partial{X}(L);\mathbb{Z})$.
A knot or link $L$ is {\it hyperbolic\/} if the interior of $X(L)$ is a 
complete hyperbolic 3-manifold of finite volume.
The exterior of a hyperbolic knot determines the knot.

The 3-component Borromean rings link $Bo$ is hyperbolic 
\cite[pages 131-132]{Th}.
There are infinitely many hyperbolic knots $K$ with trivial Alexander 
polynomial \cite{Ka04}.
Let $K_1,K_2,K_3$ be three distinct such knots, 
with exteriors $X_1,X_2$ and $X_3$.
Let
\[N=X(Bo)\cup_{i=1}^3X_i,
\]
where the longitudes and meridians of $K_i$ are identified with
those of the $i$th component of $Bo$,  for $i=1,2,3$.
The 3-manifold $N$ clearly has a geometric decomposition
into four hyperbolic pieces.
Any self-homeomorphism $h$ of $N$ is isotopic to one which
preserves this decomposition and so fixes these pieces setwise, 
since no two are homeomorphic.

Let $\nu=\pi_1N$. 
The images of the meridians give a preferred basis (up to signs)
for $\nu^{ab}=\Z^3$.
Self-homeomorphisms of $N$ preserve this basis (up to sign), 
since they must fix the pieces of the geometric decomposition.

In the first example $\pi\cong\Z^2$.
Let $\xi_i:\nu\to\Z^2$ be the epimorphism with kernel generated by 
$[\nu,\nu]$ and the meridians for $K_i$.
Then the three epimorphisms $\xi_i$ are pairwise inequivalent under the action of 
automorphisms of $\pi_1N$.
There is a compact aspherical 4-manifold $M_i$ with $\pi_1M_i\cong\Z$,
$\partial{M_i}\cong{N}$ and $\pi_1\inc=\xi_i$,
by Theorem \ref{B}. 
Thus we have three distinct 4-manifolds $M_i$ such that $\pi_1M_i\cong\Z^2$
and  $\partial{M_i}\cong{N}$.

In the other two examples $\pi\cong\Z^3$.
Let $Y=N\sqcup{T^3}$, let $f:N\to{T^3}$ be a map which induces the abelianization 
$ab:\pi_1(N)\to\mathbb{Z}^3$ and let $X$ be the 
mapping cylinder of the map $F:Y\to{T^3}\times[0,1]$ given by 
$F(n)=(f(n),0)$ for all $n\in{N}$ and $F(z)=(z,1)$ for all $z\in{T^3}$.
Then $(X,Y)$ is a $PD_4$-pair,
and so there is a compact 4-manifold $W\simeq{T^3}\times[0,1]$ with boundary $\partial{W}=N\sqcup{T^3}$, 
by Theorem \ref{PD4pair}.
Let $\partial_0W=N$ and $\partial_1W=T^3$.
(The preferred basis of $H_1(W)\cong{H_1(N)}$
determines an identification of $\partial_1W$ with $T^3$.)

The second example has connected boundary.

Let $U$ be the mapping cylinder of the orientation cover of the Klein bottle,
and let $M=U\times{S^1}$. 
Then $M$ is orientable,
$\partial{M}\cong{T^3}$ and $\pi=\pi_1M\cong\pi_1(Kb)\times\mathbb{Z}$.
Let $g:\partial_1W=T^3\to\partial{M}$ be a homeomorphism, 
and let $M(g)=W\cup_gM$.
If $h$ is another such homeomorphism then $M(g)\cong{M(h)}$ 
if and only the image of $h^{-1}g$ in $GL(3,\mathbb{Z})$ is in 
the subgroup generated by the image of $\Aut(\pi)$ 
(which contains the diagonal matrices).
This is a proper subgroup of $GL(3,\mathbb{Z})$.
Thus there are compact aspherical orientable 4-manifolds $M(g)$ and $M(h)$ 
such that $(M(g),N)$ and $(M(h),N)$ are not homotopy equivalent as pairs,
although $M(g)\simeq{M(h)}\simeq{Kb}\times{S^1}$ and 
$\partial{M(g)}\cong\partial{M(h)}\cong{N}$.

The third example has two boundary components.

Let $M(g)=W\cup_gW$ be the union of two copies of $W$ along 
$\partial_1W=T^3$, 
via a linear homeomorphism $g\in{GL(3,\mathbb{Z})}$.
There is an obvious homeomorphism $M(g^{-1})\cong{M(g)}$
which swaps the copies of $W$.
With this in mind, 
we see that $M(h)\cong{M(g)}$ if and only if $h=g^{\pm1}\delta$,
for some  $\delta$ in the diagonal subgroup of $GL(3,\mathbb{Z})$.
It follows easily that there are compact aspherical orientable 4-manifolds 
$M(g)$ and $M(h)$ such that $(M(g),\partial{M(g)})$ and 
$(M(h),\partial{M(h)})$ are not homotopy equivalent as pairs,
although $M(g)\simeq{M(h)}\simeq{T^3}$ and 
$\partial{M(g)}\cong\partial{M(h)}\cong2N$.
\end{example}

\medskip
A given 3-manifold may bound more than one compact aspherical 4-manifold.
For instance,  the 3-torus $T^3$ bounds each of $T^2\times{D^2}$ and the mapping
cylinder of the orientable double cover of $Kb\times{S^1}$.
More generally, let $N_q$ be the total space of the $S^1$-bundle over $T$
with Euler number $q$, and let $\Gamma_q=\pi_1N$.
Then $N$ bounds the corresponding disc bundle space, with $\pi\cong\mathbb{Z}^2$,
and also bounds the mapping cylinder of a double cover of $N_{2q}$,
with $\pi\cong\Gamma_{2q}$.

\section{other good groups}

In this final section we consider two possible extensions of our work.
Firstly, we might relax the hypothesis that $\pi$ be elementary amenable.
Secondly, we might ask what is known about the smooth classification.
We shall comment on each of these.

It is easy to give examples of aspherical 4-manifolds whose fundamental groups
have nonabelian free groups.
The simplest are the products of two hyperbolic surfaces.
However,  as mentioned in the introduction, 
we do not know of any examples of a compact aspherical 4-manifold $M$ 
such that $\pi_1(M)$ satisfies the DEC but is not elementary amenable.

It is plausible that the DEC might hold for all amenable groups, 
since the strategy of Freedman and Teichner in \cite{FT95} 
and the definition of amenable group each involve notions of controlled growth.
(A more speculative hope is that the DEC should hold for the broader class
of groups which have no noncyclic free subgroups.)
S. Fisher has asked whether a group $G$ with a finite 2-dimensional $K(G,1)$ complex
and such that $\beta_2^{(2)}(G)=0$ must be coherent  \cite{Fi24}.
The $L^2$-Betti numbers of a finitely generated infinite amenable group are all 0
\cite{CG86}.
Thus a positive answer to Fisher's question would suggest that amenable groups 
of type $F$ and cohomological dimension 2 should be coherent.
All such groups are Baumslag-Solitar groups $BS(1,m)$,
by \cite[Corollary 2.6.1]{HiF}, 
and so the extension to amenable groups would give no new examples with $\cd\pi=2$.

The picture is less clear for groups of cohomological dimension 3,
although there is strong evidence that amenable $PD_3$-groups are solvable.
In particular, all such 3-manifold groups are polycyclic.
There are no known examples of groups of type $F$ which have 
no noncyclic free subgroups but are not virtually solvable.

If we work on the level of $PD_n$-pairs and homotopy type then 
we need no constraints on the fundamental group $\pi$.
In \cite{dh3} we shall use the Unique Factorization Theorem for $PD_3$-complexes,
the Kurosh Theorem on subgroups of free products and a result on the end modules
$H^1(\pi;\Z\pi)$ to show that if $(X,Y)$ is a $PD_4$-pair with $X$ aspherical and
either $\cd\pi\leq2$ or $\pi$ is a duality group of dimension 3 
then $(X,Y)$ is homotopy equivalent to a boundary connected sum of such pairs 
with indecomposable boundaries.

When $M$ is a closed aspherical 4-manifold and $\pi$ is elementary amenable then
$\chi(M)=0$ and $\pi$ is polycyclic, and $M$ is smoothable.
We have no general results on the existence of smooth structures 
on  compact aspherical 4-manifolds with boundary.   
They don't always exist.
Every integral homology 3-sphere bounds a compact contractible 4-manifold,
which is determined up to homeomorphism by its boundary.
However Freedman's contractible 4-manifold with boundary the Poincar\'e homology
sphere is not smoothable, by Rokhlin's Theorem.
In general,
the existence of a smooth structure  on such contractible 4-manifolds
is a notoriously difficult question.

The Borel Uniqueness Conjecture is unknown for any closed aspherical 4-manifold with nonelementary amenable fundamental group.   Recently a counterexample to the smooth analogue of the Borel Uniqueness Conjecture in dimension 4 has been given \cite{DHHRS}, although there is no such example known with elementary amenable fundamental group.

\medskip
\noindent{\bf Acknowledgment.}
This collaboration began at the conference on
 {\it Topology of Manifolds : Interactions between high and low dimensions\/} 
 held at Creswick, VIC.
The authors would like to thank the 
 MATRIX Institute for its support.  
 JFD would like to thank the National Science Foundation for its support under grant DMS 1615056 and the Simons Collaboration Grant 713226.   We would like to thank Mark Powell for explaining the paper \cite{CP} to us.
 
 We would like to thank  Rostislav Grigorchuk for useful correspondence, in which he was asked if  there was a finitely presented group with intermediate growth with finite cohomological dimension  and replied ``I am sure that such an example is unknown."  He did write, however, that it might be possible to construct a finitely presented torsion-free group with intermediate growth.
 
 \bibliography{dh}{}
\bibliographystyle{alpha}

\end{document}